\documentclass{article}

\usepackage{floatrow}
\input xy
\xyoption{all}
\usepackage{extarrows}
\usepackage{tikz}
\usetikzlibrary{arrows}
\usepackage{comment}

\usepackage{amsmath,stmaryrd}
\usepackage{amsthm}
\usepackage{authblk}
\usepackage{url}

\theoremstyle{plain}
    \newtheorem{lemma}{Lemma}[section]
    \newtheorem{theorem}[lemma]{Theorem}

\theoremstyle{remark}

\newcommand{\byref}[1]{\overset{\textup{\ref{#1}}}{=}}

\title{Varieties of Lazy Magmas Characterized by Forbidden Substructure Theorems}

\author[1,2]{Jo\~ao Ara\'ujo\thanks{\texttt{jaraujo@ptmat.fc.ul.pt}}}
\author[6]{Fernando Maia Ferreira\thanks{\texttt{fernando.ferreira@cedis.pt}}}
\author[1,5]{Michael Kinyon\thanks{\texttt{mkinyon@du.edu}}}

\affil[1]{\small{Centro de Matem\'{a}tica e Aplica\c{c}\~{o}es, Faculdade de Ci\^{e}ncias e Tecnologia,
    Universidade Nova de Lisboa, Campus da Caparica 2829-516, Caparica, PT}}
\affil[5]{\small{Department of Mathematics, University of Denver, Denver CO 80208, USA}}
\affil[6]{\small{Department of Mathematics, Universidade Aberta, Portugal}}

\date{}
\begin{document}
\maketitle
\tableofcontents
\begin{abstract}
A magma (or groupoid) is a set with a binary operation $(A,f)$. Roughly speaking, a magma is said to be lazy if compositions  such as $f(x,f(f(y,z),u))$ depend on at most two variables. Recently,   Kaprinai, Machida and Waldhauser described the lattice of all the varieties of lazy groupoids.

A forbidden structure theorem is one that charcaterizes a smaller class $A$ inside a larger class $B$ as all the elements in $B$ that avoid some substructures. For example, a lattice is distributive (smaller class $A$) if and only if it is a lattice (larger class $B$) and avoids the pentagon and the diamond.

In this paper we provide a characterization of all pairs of lazy groupoid varieties $A\le B$ by forbidden substructure theorems. Some of the results are straightforward, but some other are very involved. All of these results and proofs were found using a computational tool that proves theorems of this type (for many different classes of relational algebras) and that we make available to every mathematician.  

\vskip 2mm

\noindent\emph{$2010$ Mathematics Subject Classification\/}. TBA.
\end{abstract}

\section{Introduction}
Let $\mathcal{C}$ be a class of mathematical structures of some type (with well defined concepts of substructure and isomorphism) and let $\mathcal{F}$ be a subset of  $\mathcal{C}$. A structure $S$ in $\mathcal{C}$ is said to \emph{avoid} the structures in $\mathcal{F}$ if no substructure of $S$ is isomorphic to a structure in $\mathcal{F}$.
Let $\llbracket \mathcal{C}\mid \mathcal{F} \rrbracket$ denote the class of all structures $S$ in $\mathcal{C}$ which avoid the structures in $\mathcal{F}$. We consider $\mathcal{F}$ to be a class of \emph{forbidden substructures} for $\llbracket \mathcal{C}\mid \mathcal{F} \rrbracket$.

A \emph{forbidden substructure theorem} characterizes a class $\mathcal{C}_1\subseteq \mathcal{C}$ as
\[
\mathcal{C}_1=\llbracket \mathcal{C}\mid \mathcal{F} \rrbracket\,,
\]
for some family of structures  $\mathcal{F}$. Such theorems can be read as saying that \emph{a structure $S$ belongs to $\mathcal{C}_1$ if and only if $S$ belongs to $\mathcal{C}$ and avoids the structures in $\mathcal{F}$}.

Probably, the paramount forbidden substructure theorems are, for graphs,
\[
\mathcal{B}=\llbracket \mathcal{G}\mid \mathcal{O}\rrbracket,
\]where $\mathcal{B}$ denotes the bipartite graphs, $\mathcal{G}$ denotes all the graphs, and $\mathcal{O}$ denotes the cycles of odd length;   and, for lattices,
\[
\mathcal{D}=\llbracket \mathcal{L}\mid \{p,d\}\rrbracket,
\]where $\mathcal{D}$ denotes the distributive lattices, $\mathcal{L}$ denotes all the lattices, and $\{p,d\}$ is the set containing the pentagon $p$ and the diamond $d$.

Both in graph theory and in universal algebra there are many more forbidden substructures theorems. In this paper, we will consider only \emph{magmas} (also known as \emph{groupoids}), that is, sets with a single binary operation and with the usual notions of submagma and isomorphism. Although some of our magmas will be semigroups, this is not generally the case, and so to improve readability, we will use the convention that juxtaposition binds more tightly than the explicit operation. Thus, for instance, for a binary operation $\cdot$, $x(y\cdot uv)$ is the same as $x\cdot (y\cdot (u\cdot v))$.

In a magma $(M,\cdot)$, it might happen that the result of a composition such as $x(y\cdot uv)$ depends only on a proper subset of the variables that occur in it. For example, if a binary operation satisfies $xx=x$ and $x\cdot yz=xy\cdot z=xz$ (these are the \emph{rectangular bands}), then any composition depends only on two variables as we have $x_1\ldots x_n=x_1x_n$. Following \cite{Machida}, a \emph{lazy magma} is, roughly speaking, a magma $(M,\cdot)$ such that repeated compositions of $\cdot$ yield nothing new (for a more formal definition, we refer the reader to \cite{Kaprinai,Machida,Machida2}). The lattice of all varieties of lazy magmas was described in \cite{Kaprinai} and is depicted in Figure \ref{fig1}.

  \begin{figure}[h]
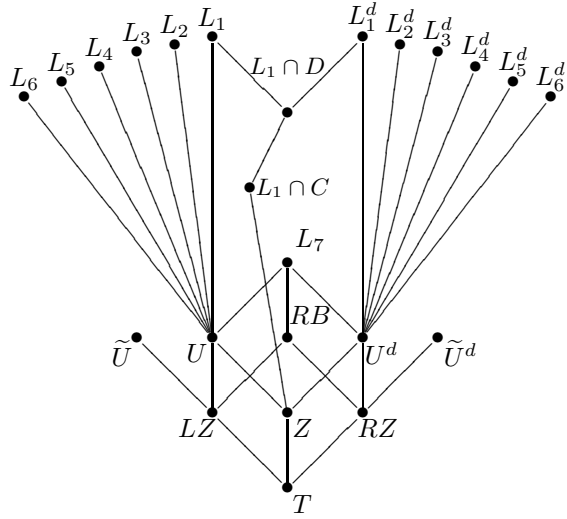

\[
\xy
(0,0)*{\bullet}="T";(2,-2)*{\mbox{$T$}};
(-10,10)*{\bullet}="LZ";(-12,8)*{\mbox{$LZ$}};
(10,10)*{\bullet}="RZ";(12,8)*{\mbox{$RZ$}};
(0,10)*{\bullet}="Z"; (2,8)*{\mbox{$Z$}};
(0,20)*{\bullet}="RB"; (3,23)*{\mbox{$RB$}};
(-10,20)*{\bullet}="U";(-12,18)*{\mbox{$U$}};
(-20,20)*{\bullet}="Utilde";(-22,18)*{\mbox{$\widetilde{U}$}};
(10,20)*{\bullet}="Ud"; (12.5,18)*{\mbox{$U^d$}};
(0,30)*{\bullet}="L7"; (3,33)*{\mbox{$L_7$}};
(-5,40)*{\bullet}="L1C"; (0.7,40)*{\small{\mbox{$L_1\cap C$}}};
(0,50)*{\bullet}="L1D"; (0,56)*{\small{\mbox{$L_1\cap D$}}};
(-10,60)*{\bullet}="L1"; (-10,62.5)*{\mbox{$L_1$}};
(-15,59)*{\bullet}="L2"; (-15,61.5)*{\mbox{$L_2$}};
(-20,58)*{\bullet}="L3"; (-20,60.5)*{\mbox{$L_3$}};
(-25,56)*{\bullet}="L4"; (-25,58.5)*{\mbox{$L_4$}};
(-30,54)*{\bullet}="L5"; (-30,56.5)*{\mbox{$L_5$}};
(-35,52)*{\bullet}="L6"; (-35,54.5)*{\mbox{$L_6$}};
%%%%EDGES%%%%%%%%%%%
{\ar@{-} "T";"LZ"};
{\ar@{-} "T";"Z"};
{\ar@{-} "T";"RZ"};
{\ar@{-} "RZ";"Ud"};
{\ar@{-} "RZ";"RB"};
{\ar@{-} "LZ";"RB"};
{\ar@{-} "LZ";"U"};
{\ar@{-} "LZ";"Utilde"};
{\ar@{-} "Z";"Ud"};
{\ar@{-} "Z";"U"};
{\ar@{-} "Z";"L1C"};
{\ar@{-} "RB";"L7"};
{\ar@{-} "Ud";"L7"};
{\ar@{-} "L1C";"L1D"};
{\ar@{-} "L1";"L1D"};
{\ar@{-} "U";"L1"};
{\ar@{-} "U";"L7"};
{\ar@{-} "U";"L1"};
{\ar@{-} "U";"L2"};
{\ar@{-} "U";"L3"};
{\ar@{-} "U";"L4"};
{\ar@{-} "U";"L5"};
{\ar@{-} "U";"L6"};
%%%%% DUAL VARIETIES%%%%%%%%
(10,60)*{\bullet}="L1d"; (10,62.8)*{\mbox{$L^d_1$}};
(15,59)*{\bullet}="L2d"; (15,61.8)*{\mbox{$L^d_2$}};
(20,58)*{\bullet}="L3d"; (20,60.8)*{\mbox{$L^d_3$}};
(25,56)*{\bullet}="L4d"; (25,58.8)*{\mbox{$L^d_4$}};
(30,54)*{\bullet}="L5d"; (30,56.8)*{\mbox{$L^d_5$}};
(35,52)*{\bullet}="L6d"; (35,54.8)*{\mbox{$L^d_6$}};
(20,20)*{\bullet}="Utilded";(23,18)*{\mbox{$\widetilde{U}^d$}};
{\ar@{-} "RZ";"Utilded"};
{\ar@{-} "Ud";"L1d"};
{\ar@{-} "Ud";"L1d"};
{\ar@{-} "Ud";"L2d"};
{\ar@{-} "Ud";"L3d"};
{\ar@{-} "Ud";"L4d"};
{\ar@{-} "Ud";"L5d"};
{\ar@{-} "Ud";"L6d"};
{\ar@{-} "L1d";"L1D"};
\endxy
\]
\caption{Lattice of varieties of lazy magmas}\label{fig1}
\end{figure}
\label{sint} where each variety is defined by the following sets of identities (omitting duals):
\begin{align*}
    xy=xz,\quad xy\cdot z=xy     \label{U}\tag*{(U)} \\
    xy=xz,\quad xy\cdot z=x      \label{tildeU}\tag*{($\widetilde{\text{U}}$)} \\
    xy\cdot z=xx,\quad x\cdot yz=xx  \label{L1}\tag*{(L$_1$)} \\
    xy\cdot z=xx,\quad x\cdot yz=xy  \label{L2}\tag*{(L$_2$)} \\
    xy\cdot z=xy,\quad x\cdot yz=xx  \label{L3}\tag*{(L$_3$)} \\
    xy\cdot z=xz,\quad x\cdot yz=xx  \label{L4}\tag*{(L$_4$)} \\
    xy\cdot z=xz,\quad x\cdot yz=xx  \label{L5}\tag*{(L$_5$)} \\
    xy\cdot z=xz,\quad x\cdot yz=xy  \label{L6}\tag*{(L$_6$)} \\
    xy\cdot z=xz,\quad x\cdot yz=xz  \label{L7}\tag*{(L$_7$)} \\
    xy=x                     \label{LZ}\tag*{(LZ)} \\
    yx=x                     \label{RZ}\tag*{(RZ)} \\
    xy=zu                    \label{Z}\tag*{(Z)} \\
    x=y                      \label{T}\tag*{(T)} \\
    x^2=x                    \label{I}\tag*{(I)} \\
    xx=yy                    \label{D}\tag*{(D)} \\
    xy=yx                    \label{C}\tag*{(C)} \\
    xy\cdot z=x\cdot yz,\quad x\cdot yx=x  \label{RB}\tag*{(RB)}
\end{align*}

The goal of this paper is the following. Let $A$ and $B$ be two varieties in the list above such that $A$ is covered by $B$. Then we provide a forbidden substructure theorem $A=\llbracket B\mid F\rrbracket$.

The main theorem of this paper is the following. (To avoid duplicating nearly 30 Cayley tables, we will give them explicitly later. The ones used repeatedly will be shown
at the beginning of {\S}\ref{Sec:proofs} and the rest will occur in the statements of the particular theorems.)

\begin{theorem}\label{main} For the varieties above, we have

  \begin{figure}[H]
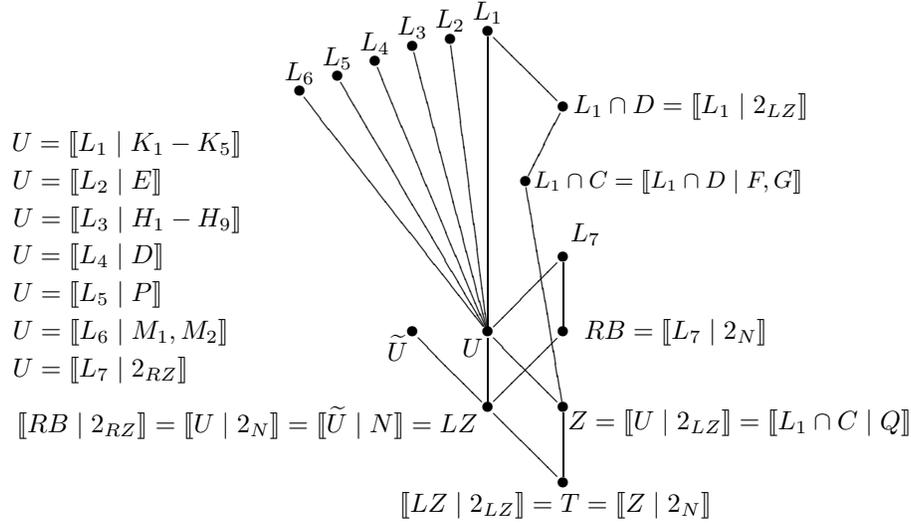

\[
\xy
(0,0)*{\bullet}="T";(-1,-3)*{\mbox{$\llbracket LZ\mid 2_{LZ} \rrbracket=T=\llbracket Z\mid 2_N \rrbracket$}};
(-10,10)*{\bullet}="LZ";(-42,8)*{\mbox{$\llbracket RB\mid  2_{RZ}\rrbracket=\llbracket U\mid 2_N\rrbracket=\llbracket \widetilde{U}\mid N \rrbracket=LZ$}};
(-58,45)*{\mbox{$U=\llbracket L_1\mid  K_1-K_5\rrbracket$}};
(-63.2,40)*{\mbox{$U=\llbracket L_2\mid  E\rrbracket$}};
(-57.9,35)*{\mbox{$U=\llbracket L_3\mid  H_1 - H_9 \rrbracket$}};
(-63.1,30)*{\mbox{$U=\llbracket L_4\mid  D\rrbracket$}};
(-63.2,25)*{\mbox{$U=\llbracket L_5\mid  P\rrbracket$}};
(-58.8,20)*{\mbox{$U=\llbracket L_6\mid  M_1 , M_2 \rrbracket$}};
(-61.5,15)*{\mbox{$U=\llbracket L_7\mid  2_{RZ}\rrbracket$}};
%(10,10)*{\bullet}="RZ";(12,8)*{\mbox{$RZ$}};
(0,10)*{\bullet}="Z"; (23.5,8)*{\mbox{$Z=\llbracket U\mid 2_{LZ} \rrbracket=\llbracket L_1\cap C \mid Q\rrbracket$}};
(0,20)*{\bullet}="RB"; (15,20)*{\mbox{$RB=\llbracket L_7\mid 2_{N} \rrbracket$}};
(-10,20)*{\bullet}="U";(-12,18)*{\mbox{$U$}};
(-20,20)*{\bullet}="Utilde";(-22,18)*{\mbox{$\widetilde{U}$}};
%(10,20)*{\bullet}="Ud"; (12.5,18)*{\mbox{$U^d$}};
(0,30)*{\bullet}="L7"; (3,33)*{\mbox{$L_7$}};
(-5,40)*{\bullet}="L1C"; (14,40)*{\small{\mbox{$L_1\cap C=\llbracket L_1\cap D \mid F,G\rrbracket$}}};
(0,50)*{\bullet}="L1D"; (17,50)*{{\mbox{$L_1\cap D=\llbracket L_1\mid 2_{LZ}\rrbracket$}}};
(-10,60)*{\bullet}="L1"; (-10,62.5)*{\mbox{$L_1$}};
(-15,59)*{\bullet}="L2"; (-15,61.5)*{\mbox{$L_2$}};
(-20,58)*{\bullet}="L3"; (-20,60.5)*{\mbox{$L_3$}};
(-25,56)*{\bullet}="L4"; (-25,58.5)*{\mbox{$L_4$}};
(-30,54)*{\bullet}="L5"; (-30,56.5)*{\mbox{$L_5$}};
(-35,52)*{\bullet}="L6"; (-35,54.5)*{\mbox{$L_6$}};
%%%%EDGES%%%%%%%%%%%
{\ar@{-} "T";"LZ"};
{\ar@{-} "T";"Z"};
%{\ar@{-} "T";"RZ"};
%{\ar@{-} "RZ";"Ud"};
%{\ar@{-} "RZ";"RB"};
{\ar@{-} "LZ";"RB"};
{\ar@{-} "LZ";"U"};
{\ar@{-} "LZ";"Utilde"};
%{\ar@{-} "Z";"Ud"};
{\ar@{-} "Z";"U"};
{\ar@{-} "Z";"L1C"};
{\ar@{-} "RB";"L7"};
%{\ar@{-} "Ud";"L7"};
{\ar@{-} "L1C";"L1D"};
{\ar@{-} "L1";"L1D"};
{\ar@{-} "U";"L1"};
{\ar@{-} "U";"L7"};
{\ar@{-} "U";"L1"};
{\ar@{-} "U";"L2"};
{\ar@{-} "U";"L3"};
{\ar@{-} "U";"L4"};
{\ar@{-} "U";"L5"};
{\ar@{-} "U";"L6"};
%%%%%% DUAL VARIETIES%%%%%%%%
%(10,60)*{\bullet}="L1d"; (10,62.8)*{\mbox{$L^d_1$}};
%(15,59)*{\bullet}="L2d"; (15,61.8)*{\mbox{$L^d_2$}};
%(20,58)*{\bullet}="L3d"; (20,60.8)*{\mbox{$L^d_3$}};
%(25,56)*{\bullet}="L4d"; (25,58.8)*{\mbox{$L^d_4$}};
%(30,54)*{\bullet}="L5d"; (30,56.8)*{\mbox{$L^d_5$}};
%(35,52)*{\bullet}="L6d"; (35,54.8)*{\mbox{$L^d_6$}};
%(20,20)*{\bullet}="Utilded";(23,18)*{\mbox{$\widetilde{U}^d$}};
%{\ar@{-} "RZ";"Utilded"};
%{\ar@{-} "Ud";"L1d"};
%{\ar@{-} "Ud";"L1d"};
%{\ar@{-} "Ud";"L2d"};
%{\ar@{-} "Ud";"L3d"};
%{\ar@{-} "Ud";"L4d"};
%{\ar@{-} "Ud";"L5d"};
%{\ar@{-} "Ud";"L6d"};
%{\ar@{-} "L1d";"L1D"};
\endxy
\]
\caption{Lattice of varieties of lazy magmas characterized by forbidden substructures}\label{fig2}
\end{figure}

The corresponding results for the dual varieties can be obtained using the dual models (that is, transposing their Cayley Tables). For example, $U^d=\llbracket L^d_5\mid  P^d\rrbracket$.
\end{theorem}

To find these results (the models, the conjectures and the proofs of the theorems) we used a computational tool that we produced and have made freely available at \cite{site}.
The tool is built on top of the automated theorem prover \textsc{Prover9} and the finite model builder \textsc{Mace4}, both developed by McCune \cite{Prover9}. The Cayley tables for the models we present in the theorems come originally from \textsc{Mace4}, and the original proofs of the theorems come from \textsc{Prover9}.

The proofs we present here have been ``humanized,'' which is to say that we carefully studied the \textsc{Prover9} proofs until we were able to extract the essential ideas.
The most dramatic simplification occurs in the case of $U=\llbracket L_3\mid \mathcal{F}\rrbracket$. The set $\mathcal{F}$ contains $9$ models, one of size $6$, four of size $5$, and four of size $4$. We originally prepared a more or less literal translation of the \textsc{Prover9} proof which ran to over seven pages and was based on rather subtle manipulations of equations. The present proof was based on realizing that what \textsc{Prover9} was revealing to us was that $2$-generated submagmas of lazy magmas are rather well-behaved. 
 
Section \ref{Sec:proofs} contains all the theorems needed to prove Theorem \ref{main}. Section \ref{tool} contains a general explanation of the computational tool we wrote to produce these forbidden substructure theorems (and many others).

\section{The Proofs}\label{Sec:proofs}

Throughout this section, all magmas will be assumed to be nonempty.

Certain forbidden submagmas will occur repeatedly in what follows, so we describe them here to avoid duplicating Cayley tables.

A \emph{left zero band} is a magma satisfying \ref{LZ}, that is, $xy=x$ for all $x,y$.
Left zero bands are clearly semigroups and also idempotent since $xx=x$.
Any nonempty subset of a left zero band is a subsemigroup which is also a left zero band.
It is easy to see that up to isomorphism, there is exactly one $2$-element left zero band, which we denote by $2_{LZ}$:
\[
\begin{array}{c|cc}
2_{LZ} & 0 & 1 \\
\hline
0 & 0 & 0 \\
1 & 1 & 1
\end{array}
\]

Dual remarks apply to \emph{right zero bands} which satisfy \ref{RZ}, that is, $xy=y$ for all $x,y$. We denote the $2$-element right zero band
by $2_{RZ}$:
\[
\begin{array}{c|cc}
2_{RZ} & 0 & 1 \\
\hline
0 & 0 & 1 \\
1 & 0 & 1
\end{array}
\]

A \emph{null semigroup} is a semigroup satisfying \ref{Z}, that is, $xy=zu$ for all $x,y,z,u$.
It is easy to see that up to isomorphism, there is exactly one $2$-element null semigroup, which we denote by $2_N$:
\[
\begin{array}{c|cc}
2_{N} & 0 & 1 \\
\hline
0 & 0 & 0 \\
1 & 0 & 0
\end{array}
\]

\subsection{$L_1\cap D$ characterized inside $L_1$}
Recall the relevant identities:
\begin{align*}
    xy\cdot z=xx,\quad x\cdot yz=xx  \tag*{(L$_1$)} \\
    xx=yy                            \tag*{(D)}
\end{align*}
Since $L_1$ is a variety of semigroups, we shall use associativity freely.

\begin{theorem}\label{Thm:L1_D_avoids_A}
A magma belongs to $L_1\cap D$ if and only if it belongs to $L_1$ and avoids the $2$-element left zero band $2_{LZ}$.
\end{theorem}
\begin{proof}
Let $M$ be a semigroup in $L_1$. For $x,y\in M$, set $S(x,y) = \{xx,yy\}$. Whether or not the elements of $S(x,y)$ are
distinct, $S(x,y)$ is a subsemigroup with Cayley table
\begin{equation}\label{Eqn:lzb_L1D}
\begin{array}{c|cc}
\cdot & xx & yy \\
\hline
xx & xx & xx \\
yy & yy & yy
\end{array}
\end{equation}
The entries of \eqref{Eqn:lzb_L1D} are easily verfied using \ref{L1}.

Now if \ref{D} fails in $M$, then there exist $a,b\in M$ such that $aa\neq bb$. In this case, $|S(a,b)| = 2$ and $S(a,b)$ is
a copy of $2_{LZ}$. Conversely, if $M$ contains a copy of $2_{LZ}$, then $0\cdot 0 = 0\neq 1\neq 1\cdot 1$ and thus
\ref{D} fails.
\end{proof}

\subsection{$RB$ characterized inside $L_7$}
Recall the relevant identities:
\begin{align*}
    xy\cdot z=xz\,,\quad x\cdot yz=xz          \tag*{(L$_7$)} \\
    xy\cdot z=x\cdot yz\,,\quad x\cdot yx=x    \tag*{(RB)}
\end{align*}
Since both $L_7$ and $RB$ are varieties of semigroups, we will use associativity freely.

\begin{theorem}
A magma belongs to $RB$ if and only if it belongs to $L_7$ and avoids the 2-element null semigroup $2_N$.
\end{theorem}
\begin{proof}
Let $M$ be a semigroup in $L_7$. For $x,y\in M$, let $S(x,y) = \{x,xyx\}$. Whether or not the elements of $S(x,y)$ are
distinct, $S(x,y)$ is a subsemigroup of $M$ and is null:
\begin{equation}\label{Eqn:rb_L7}
\begin{array}{c|cc}
\cdot & x & xyx \\
\hline
  x & xyx & xyx \\
xyx & xyx & xyx
\end{array}
\end{equation}
The entries of \eqref{Eqn:lzb_L1D} are easily confirmed using \ref{L7}.

Now if \ref{RB} fails in $M$, then there exist $a,b\in M$ such that $aba\neq a$. In this case, $|S(a,b)|=2$ and
$S(a,b)$ is a copy of $2_N$. Conversely, if $M$ contains a copy of $2_N$, then $1\cdot 0\cdot 1 = 0\neq 1$, and thus
\ref{RB} fails.
\end{proof}

\subsection{$Z$ characterized inside $U$}
Recall the relevant identities:
\begin{align*}
   xy=xz\,,\quad xy\cdot z=xy   \tag*{(U)} \\
   xy=zu                        \tag*{(Z)}
\end{align*}

\begin{theorem}
A magma belongs to $Z$ if and only if it belongs to $U$ and avoids the $2$-element left zero band $2_{LZ}$.
\end{theorem}
\begin{proof}
Let $M$ be a magma in $U$. For $x,y,z,u\in M$, let $S(x,y,z,u) = \{ xy,zu\}$. Whether or not the elements of $S(x,y,z,u)$ are
distinct, $S(x,y,z,u)$ is a submagma of $M$ with Cayley table
\begin{equation}\label{Eqn:lz_ZU}
\begin{array}{c|cc}
\cdot & xy & zu \\
\hline
xy & xy & zu \\
zu & zu & zu
\end{array}
\end{equation}
The entries of \eqref{Eqn:lz_ZU} are easily verified using \ref{U}.

Now if \ref{Z} fails in $M$, then there exist $a,b,c,d\in M$ such that $ab\neq cd$. In this case, $|S(a,b,c,d)| = 2$ and $S(a,b,c,d)$
is a copy of $2_{LZ}$. Conversely, if $M$ contains a copy of $2_{LZ}$, then $0\cdot 0 = 0\neq 1 = 1\cdot 0$, and thus \ref{Z} fails.
\end{proof}

\subsection{$Z$ characterized inside $L_1\cap C$}
Recall the relevant identities
\begin{align}
  xy\cdot z = xx\,,\quad  x\cdot yz = xx\,, \tag*{(L$_1$)} \\
  xy=yx \tag*{(C)} \\
  xy=zu \tag*{(Z)}
\end{align}
Since $L_1$ is a variety of semigroups, we use associativity freely. Recall that $L_1\cap C$
lies inside $D$ so that $xx=yy$ for all $x,y$. We let $0$ denote this constant value $xx$, which is a
zero element for any semigroup in $L_1\cap D$.

Let $Q = \{0,1,2,3\}$ be the magma defined by $1\cdot 2 = 2\cdot 1 = 3$ with all other products equaling $0$.
It is easy to see that $Q$ belongs to $L_1\cap C$.

\begin{theorem}\label{Thm:L1C_Z}
A magma belongs to $Z$ if and only if it belongs to $L_1\cap C$ and avoids the magma $Q$ with Cayley table
\[
\begin{array}{c|cccc}
Q & 0 & 1 & 2 & 3 \\
\hline
0 & 0 & 0 & 0 & 0 \\
1 & 0 & 0 & 3 & 0 \\
2 & 0 & 3 & 0 & 0 \\
3 & 0 & 0 & 0 & 0
\end{array}
\]
\end{theorem}
\begin{proof}
Let $M$ be a semigroup in $L_1\cap C$. For $x,y\in M$, let $S(x,y) = \{0,x,y,xy\}$. Whether or not the elements of $S(x,y)$
are distinct, $S(x,y)$ is a subsemigroup of $M$ with Cayley table
\begin{equation}\label{Eqn:Q_ZL1C}
\begin{array}{c|cccc}
\cdot & 0 & x & y & xy \\
\hline
 0 & 0 &  0 &  0 & 0 \\
 x & 0 &  0 & xy & 0 \\
 y & 0 & xy &  0 & 0 \\
xy & 0 &  0 &  0 & 0
\end{array}
\end{equation}
The entries of \eqref{Eqn:Q_ZL1C} are easily verified using \ref{L1} and \ref{C}.

If \ref{Z} fails in $M$, then there exist $a,b\in M$ with $ab\neq 0$. In this case $|S(a,b)| = 4$ and $S(a,b)$ is a copy of $Q$.
Conversely, if $M$ contains a copy of $Q$, then $0\cdot 0 = 0\neq 3= 1\cdot 2$, and thus \ref{Z} fails.
\end{proof}

\subsection{$LZ$ characterized inside $\widetilde{U}$}
Recall the relevant identities:
\begin{align*}
    xy=xz\,,\quad xy\cdot z=x   \tag*{($\widetilde{\text{U}}$)} \\
    xy=x                        \tag*{(LZ)}
\end{align*}

\begin{theorem}
A magma belongs to $LZ$ if and only if it belongs to $\widetilde{U}$ and avoids the magma $N$ with Cayley table
\[
\begin{array}{c|cc}
N & 0 & 1 \\
\hline
0 & 1 & 1 \\
1 & 0 & 0
\end{array}
\]
\end{theorem}
\begin{proof}
Let $M$ be a magma in \ref{tildeU}. For $x,y\in M$, let $S(x,y) = \{x,xy\}$. Whether or not the elements of $S(x,y)$ are
distinct, $S(x,y)$ is a submagma of $M$ with Cayley table
\begin{equation}\label{Eqn:N_LZ}
\begin{array}{c|cc}
\cdot & x & xy \\
\hline
 x & xy & xy \\
xy &  x &  x
\end{array}
\end{equation}
The entries of \eqref{Eqn:N_LZ} are easily verified using \ref{tildeU}.

If \ref{LZ} fails in $M$, then there exist $a,b\in M$ such that $ab\neq a$. In this case $|S(a,b)|=2$ and $S(a,b)$ is a copy of $N$.
Conversely, if $M$ contains a copy of $N$, then $0\cdot 1 = 1\neq 0$ and thus \ref{LZ} fails.
\end{proof}

\subsection{$LZ$ characterized inside $RB$}
Recall the relevant identities:
\begin{align*}
    xy\cdot z = x\cdot yz\,,\quad   x\cdot yx = x   \tag*{(RB)} \\
    xy=x                                            \tag*{(LZ)}
\end{align*}
Since both $RB$ and $LZ$ are varieties of semigroups, we use associativity freely.

\begin{theorem}
A magma belongs to $LZ$ if and only if it belongs to $RB$ and avoids the $2$-element right zero band $2_{RZ}$.
\end{theorem}
\begin{proof}
Let $M$ be a semigroup in $RB$. Note that for all $x\in M$, $xx\byref{RB} x(xxx) = x\cdot xx\cdot x\byref{RB} x$, so that every element
of $M$ is an idempotent. For $x,y\in M$, let $S(x,y) = \{x,xy\}$. Whether or not the element of $S(x,y)$ are
distinct, $S(x,y)$ is a subsemigroup of $M$ with Cayley table
\begin{equation}\label{Eqn:LZ_RB}
\begin{array}{c|cc}
\cdot & x & xy \\
\hline
 x &  x & xy \\
xy &  x & xy
\end{array}
\end{equation}
The off-diagonal entries of \eqref{Eqn:LZ_RB} are easily verified using \ref{RB}.

If \ref{LZ} fail in $M$, then there exist $a,b\in M$ such that $ab\neq a$. In this case $|S(a,b)|=2$ and $S(a,b)$ is a copy
of $2_{RZ}$. Conversely, if $M$ contains a copy of $2_{RZ}$, then $0\cdot 1 = 1\neq 0$ and thus \ref{LZ} fails.
\end{proof}

\subsection{$LZ$ characterized inside $U$}
Recall the relevant identities:
\begin{align*}
   xy=xz,\quad xy\cdot z=xy \tag*{(U)} \\
   xy=x                     \tag*{(LZ)}
\end{align*}

\begin{theorem}
A magma belongs to $LZ$ if and only if it belongs to $U$ and avoids the $2$-element null semigroup $2_N$.
\end{theorem}
\begin{proof}
Let $M$ be a magma in $U$. For $x,y\in M$, let $S(x,y) = \{x,xy\}$. Whether or not the elements of $S(x,y)$ are distinct, $S(x,y)$ is a submagma of $M$
with Cayley table
\begin{equation}\label{Eqn:U_LZ}
\begin{array}{c|cc}
\cdot & x & xy \\
\hline
 x &  xy & xy \\
xy &  xy & xy
\end{array}
\end{equation}
The entries of \eqref{Eqn:U_LZ} are easily verified using \ref{U}.

If \ref{LZ} fails in $M$, then there exist $a,b\in M$ such that $ab\neq a$. In this case $|S(a,b)|=2$ and $S(a,b)$ is a copy of $2_N$.
Conversely, if $M$ contains a copy of $2_N$, then $1\cdot 0 = 0\neq 1$ and thus \ref{LZ} fails.
\end{proof}

\subsection{$T$ characterized inside $LZ$}
Recall the relevant identities:

\begin{align*}
    xy=x                    \tag*{(LZ)} \\
    x=y                     \tag*{(T)}
\end{align*}

\begin{theorem}
A magma belongs to $T$ if and only if it belongs to $LZ$ and avoids the $2$-element left zero band $2_{LZ}$.
\end{theorem}
\begin{proof}
Let $M$ be a magma in $LZ$. By \ref{LZ}, any nonempty subset of $M$ forms a submagma of $M$.
If \ref{T} fails in $M$, then $|M|\geq 2$ and there exist a $2$-element submagma of $M$ which is a copy of $2_{LZ}$.
Conversely, if $M$ contains a copy of $2_{LZ}$, then $|M|\geq 2$ and so evidently \ref{T} fails.
\end{proof}

\subsection{$T$ characterized inside $Z$}
Recall the relevant identities:
\begin{align*}
    xy=zu    \tag*{(Z)} \\
    x=y      \tag*{(T)}
\end{align*}

\begin{theorem}
A magma belongs to $T$ if and only if it belongs to $Z$ and avoids the $2$-element null semigroup $2_N$.
\end{theorem}
\begin{proof}
Let $M$ be a magma in $Z$, and let $0\in M$ be the element such that $xy=0$ for all $x,y\in M$.
For $x\in M$, $S(x) = \{0,x\}$ is evidently a submagma and a null semigroup. Thus $M$ avoids $2_N$
if and only if $S(x)$ is a singleton, that is, if and only if \ref{T} holds for all $x,y\in M$.
\end{proof}

\subsection{$L_1\cap C$ characterized inside $L_1\cap D$}
Recall the relevant identities
\begin{align*}
  xy\cdot z = xx\,, \quad x\cdot yz = xx  \tag*{(L$_1$)} \\
  xx=yy \tag*{(D)} \\
  xy=yx \tag*{(C)}
\end{align*}
Since $L_1$ is a variety of semigroups, we use associativity freely. For a magma in $D$, let $0$ denote the element such that $xx=0$ for all $x$. For a semigroup $M$ in $L_1\cap D$, $0$ is a zero element. Indeed, $0x = xxx \byref{L1} xx  0$ and $x0 = xxx \byref{L1} xx = 0$.

For a semigroup in $L_1\cap C$, note that $xx\byref{L1} xyx\byref{C} yxx\byref{L1} yy$, so that \ref{D} holds. Thus $L_1\cap C\subseteq L_1\cap D$.

\begin{theorem}\label{Thm:L1D_C}
A semigroup in $L_1\cap D$ belongs to $L_1\cap C$ if and only if it belongs to $L_1\cap D$ and avoids both of the magmas $F$ and $G$ with the following Cayley tables
\[
\begin{array}{c|cccc}
F & 0 & 1 & 2 & 3 \\
\hline
0 & 0 & 0 & 0 & 0 \\
1 & 0 & 0 & 3 & 0 \\
2 & 0 & 0 & 0 & 0 \\
3 & 0 & 0 & 0 & 0
\end{array}
\qquad
\begin{array}{c|ccccc}
G & 0 & 1 & 2 & 3 & 4\\
\hline
0 & 0 & 0 & 0 & 0 & 0\\
1 & 0 & 0 & 3 & 0 & 0\\
2 & 0 & 4 & 0 & 0 & 0\\
3 & 0 & 0 & 0 & 0 & 0\\
4 & 0 & 0 & 0 & 0 & 0
\end{array}
\]
\end{theorem}
\begin{proof}
Let $M$ be a semigroup in $L_1\cap D$. For $x,y\in M$, let $S(x,y) = \{0,x,y,xy,yx\}$. Whether or not the elements of $S(x,y)$ are distinct, $S(x,y)$ is a subsemigroup of
$M$ with Cayley table
\begin{equation}\label{Eqn:L1D_C}
\begin{array}{c|ccccc}
\cdot & 0 & x & y & xy & yx \\
\hline
 0 &  0 &  0 &  0 &  0 &  0 \\
 x &  0 &  0 & xy &  0 &  0 \\
 y &  0 & yx &  0 &  0 &  0 \\
xy &  0 &  0 &  0 &  0 &  0 \\
yx &  0 &  0 &  0 &  0 &  0
\end{array}
\end{equation}
The entries of \eqref{Eqn:L1D_C} are easily verified using \ref{L1} and \ref{D}.

If \ref{C} fails in $M$, then there exist $a,b\in M$ such that $ab\neq ba$. If $a=ab$, then $ba=bab=bb=0$ and $ab=abb=aa=0$, a contradiction. If $b=ab$, then $ab=aab=aa=0$ and
$ba=aba=aa=0$, again a contradiction. Similarly, $a\neq ba$ and $b\neq ba$. Thus $|S(a,b)|\geq 4$. Hence there are two possibilities: either (say) $ba = 0$ and $ab\neq 0$ or else neither $ab$ nor $ba$ equal $0$. In the former case, $S(a,b) = \{0,a,b,ab\}$ is a copy of $F$. In the latter case, $S(a,b) = \{0,a,b,ab,ba\}$ is a copy of $G$.

Conversely, if $M$ contains a copy of $F$ or a copy of $G$, then in either case, $1\cdot 2\neq 2\cdot 1$ and thus \ref{C} fails.
\end{proof}

\subsection{$U$ characterized inside $L_1$}
Recall the relevant identities
\begin{align*}
  xy\cdot z = xx\,, \quad x\cdot yz = xx \tag*{(L$_1$)} \\
  xy = xz\,, \quad  xy\cdot z = xy\,.    \tag*{(U)}
\end{align*}
Since $L_1$ is a variety of semigroups, we will use associativity freely.
To prove that a semigroup in $L_1$ is in $U$, it is sufficient to prove $xy=xz$
for all $x,y,z$, which is equivalent to $xy=xx$ for all $x,y$.

\begin{theorem}\label{Thm:L1_U}
A magma belongs to $U$ if and only if it belongs to $L_1$ and avoids each of the magmas $K_1$--$K_5$ with the following
Cayley tables:
\[
\begin{array}{c|cccc}
K_1 & 0 & 1 & 2 & 3 \\
\hline
0 & 2 & 2 & 2 & 2 \\ % 2,2,2,2,
1 & 3 & 2 & 2 & 2 \\ % 3,2,2,2,
2 & 2 & 2 & 2 & 2 \\ % 2,2,2,2
3 & 2 & 2 & 2 & 2    % 2,2,2,2
\end{array}
\qquad
\begin{array}{c|cccc}
K_2 & 0 & 1 & 2 & 3 \\
\hline
0 & 2 & 3 & 3 & 3 \\ % 2,3,2,2,
1 & 3 & 2 & 2 & 2 \\ % 3,2,2,2,
2 & 2 & 2 & 2 & 2 \\ % 2,2,2,2
3 & 2 & 2 & 2 & 2    % 2,2,2,2
\end{array}
\]

\[
\begin{array}{c|ccccc}
K_3 & 0 & 1 & 2 & 3 & 4 \\
\hline
0 & 2 & 4 & 2 & 2 & 2 \\ % 2,4,2,2,2
1 & 3 & 2 & 2 & 2 & 2 \\ % 3,2,2,2,2
2 & 2 & 2 & 2 & 2 & 2 \\ % 2,2,2,2,2
3 & 2 & 2 & 2 & 2 & 2 \\ % 2,2,2,2,2
4 & 2 & 2 & 2 & 2 & 2    % 2,2,2,2,2
\end{array}
\qquad
\begin{array}{c|ccccc}
K_4 & 0 & 1 & 2 & 3 & 4 \\
\hline
0 & 4 & 4 & 4 & 4 & 4 \\ % 4,4,4,4,4
1 & 3 & 2 & 2 & 2 & 2 \\ % 3,2,2,2,2
2 & 2 & 2 & 2 & 2 & 2 \\ % 2,2,2,2,2
3 & 2 & 2 & 2 & 2 & 2 \\ % 2,2,2,2,2
4 & 4 & 4 & 4 & 4 & 4    % 4,4,4,4,4
\end{array}
\]

\[
\begin{array}{c|cccccc}
K_5 & 0 & 1 & 2 & 3 & 4 & 5 \\
\hline
0 & 5 & 4 & 5 & 5 & 5 & 5 \\ % 5,4,5,5,5,5
1 & 3 & 2 & 2 & 2 & 2 & 2 \\ % 3,2,2,2,2,2
2 & 2 & 2 & 2 & 2 & 2 & 2 \\ % 2,2,2,2,2,2
3 & 2 & 2 & 2 & 2 & 2 & 2 \\ % 2,2,2,2,2,2
4 & 5 & 5 & 5 & 5 & 5 & 5 \\ % 5,5,5,5,5,5
5 & 5 & 5 & 5 & 5 & 5 & 5    % 5,5,5,5,5,5
\end{array}
\]
\end{theorem}
\begin{proof}
Let $M$ be a semigroup in $L_1$. For $x,y\in M$, let $S(x,y) = \{y,x,xx,xy,yx,yy\}$. Whether or not the elements of $S(x,y)$ are distinct, $S(x,y)$ is a subsemigroup of
$M$ with Cayley table
\begin{equation}\label{Eqn:L1_U}
\begin{array}{c|cccccc}
\cdot & y & x & xx & xy & yx & yy \\
\hline
 y & yy & yx & yy & yy & yy & yy \\
 x & xy & xx & xx & xx & xx & xx \\
 x & xx & xx & xx & xx & xx & xx \\
xy & xx & xx & xx & xx & xx & xx \\
yx & yy & yy & yy & yy & yy & yy \\
yy & yy & yy & yy & yy & yy & yy
\end{array}
\end{equation}
The entries of \eqref{Eqn:L1_U} are easily verified using \ref{L1}.

If \ref{U} fails in $M$, then there exist $a,b\in M$ such that $ab\neq aa$. We now show that $a$ and $b$ do not coincide with any other element of $S(a,b)$.
Clearly $a\neq b$. If $b = uv$ for some $u,v\in M$, then $ab = auv = aa$ by \ref{L1}. If $a = au$ for some $u\in M$, then $ab = aub = aa$.
If $a = bu$ for some $u\in M$, than $ab = bub = bb$ by \ref{L1} and so $abw = bbw$ for all $w\in M$, that is, $aa = bb$; hence $ab = bb = aa$.
Each case contradicts the assumption that $ab\neq aa$. We may also conclude that $a$, $b$, $ab$, and $aa$ are all distinct.

What remains is to consider the different cases of elements of $S(a,b)$ coinciding. If no two elements coincide, then an examination of \eqref{Eqn:L1_U}
(with $x=a$, $y=b$) shows that $S(a,b)$ is a copy of $K_5$ with $b=0$, $a=1$, $aa=2$, $ab=3$, $ba=4$, $bb=5$.

If $aa=bb$ but no other elements coincide, then \eqref{Eqn:L1_U} becomes
\[
\begin{array}{c|ccccc}
\cdot & b & a & aa & ab & ba \\
\hline
 b & aa & ba & aa & aa & aa \\
 a & ab & aa & aa & aa & aa \\
aa & aa & aa & aa & aa & aa \\
ab & aa & aa & aa & aa & aa \\
ba & aa & aa & aa & aa & aa
\end{array}
\]
This is evidently a copy of $K_3$ with $b=0$, $a=1$, $aa=3$, $ab=4$, $aa=5$.

If $ba=bb$ but no other elements coincide, then \eqref{Eqn:L1_U} becomes
\[
\begin{array}{c|ccccc}
\cdot & b & a & aa & ab & ba \\
\hline
 b & ba & ba & ba & ba & ba \\
 a & ab & aa & aa & aa & aa \\
aa & aa & aa & aa & aa & aa \\
ab & aa & aa & aa & aa & aa \\
ba & ba & ba & ba & ba & ba
\end{array}
\]
This is evidently a copy of $K_4$ with $b=0$, $a=1$, $aa=3$, $ab=4$, $aa=5$.

If $aa=ba$, then $aau=bau$ for all $u\in M$, and thus $aa=bb$ by \ref{L1}. In this case, \eqref{Eqn:L1_U} reduces to
\[
\begin{array}{c|cccc}
\cdot & b & a & aa & ab \\
\hline
 b & aa & aa & aa & aa \\
 a & ab & aa & aa & aa \\
aa & aa & aa & aa & aa \\
ab & aa & aa & aa & aa
\end{array}
\]
This is evidently a copy of $K_1$ with  $b=0$, $a=1$, $aa=3$, $ab=4$.

If $ab=ba$, then $abu=bau$ for all $u\in M$, and thus $aa=bb$ by \ref{L1}. In this case , \eqref{Eqn:L1_U} reduces to
\[
\begin{array}{c|cccc}
\cdot & b & a & aa & ab \\
\hline
 b & aa & ab & aa & aa \\
 a & ab & aa & aa & aa \\
aa & aa & aa & aa & aa \\
ab & aa & aa & aa & aa
\end{array}
\]
This is evidently a copy of $K_2$ with  $b=0$, $a=1$, $aa=3$, $ab=4$.

This exhausts all possibilities. Thus if \ref{U} fails in $M$, then $M$ contains a copy of at least one of $K_1$--$K_5$.
Conversely, if $M$ contains a copy of at least one of $K_1$--$K_5$, then in all $5$ cases, $1\cdot 0 = 3\neq 2 = 1\cdot 1$,
and so \ref{U} fails.
\end{proof}

\subsection{$U$ characterized inside $L_2$}
Recall the relevant identities:
\begin{align*}
  xy\cdot z = xx\,, \quad  x\cdot yz = xy \tag*{(L$_2$)} \\
  xy = xz\,, \quad xy\cdot z = xy         \tag*{(U)}
\end{align*}
To show that a magma in $L_2$ is in $U$, it is sufficient to show $xy=xz$ for all $x,y,z$, which is equivalent to $xy=xx$ for all $x,y$.

\begin{theorem}\label{Thm:L2_U}
A magma belongs to $U$ if and only if it belongs to $L_2$ and avoids the magma $E$ with the following Cayley table
\[
\begin{array}{c|cccc}
E & 0 & 1 & 2 & 3 \\
\hline
0 & 2 & 3 & 2 & 2 \\
1 & 1 & 1 & 1 & 1 \\
2 & 2 & 2 & 2 & 2 \\
3 & 2 & 2 & 2 & 2
\end{array}
\]
\end{theorem}
\begin{proof}
Let $M$ be a magma in $L_2$. For $x,y\in M$, let $S(x,y) = \{x,yx,xx,xy\}$. Whether or not the elements of $S(x,y)$ are distinct, $S(x,y)$
is a submagma of $M$ with Cayley table
\begin{equation}\label{Eqn:E_L2U}
\begin{array}{c|cccc}
E & x & yx & xx & xy \\
\hline
 x & xx & xy & xx & xx \\
yx & yx & yx & yx & yx \\
xx & xx & xx & xx & xx \\
xy & xx & xx & xx & xx
\end{array}
\end{equation}
The entries in \eqref{Eqn:E_L2U} are easily verified using \ref{L2}.

If \ref{L2} fails in $M$, then there exist $a,b\in M$ such that $ab\neq aa$. We now show that $|S(a,b)|=4$. Indeed, if $a=ba$, then
$aa=a\cdot ba\byref{L2} ab$. If $a=au$ for $u\in \{a,b\}$, then $ab = au\cdot b\byref{L2} aa$. If $ba=au$ for $u\in \{a,b\}$, then
$a\cdot ba = a\cdot au$, and so by \ref{L2}, $ab=aa$. Each case contradicts $ab\neq aa$, and so $|S(a,b)|=4$ as claimed. In this
case, $S(a,b)$ is a copy of $E$.

Conversely, if $M$ contains a copy of $E$, then $0\cdot 0 = 2\neq 3 = 0\cdot 1$, and thus \ref{L2} fails.
\end{proof}

\subsection{$U$ characterized inside $L_3$}
Recall the relevant identities:
\begin{align*}
  xy\cdot z = xy\,, \quad x\cdot yz = xx\,. \tag*{(L$_3$)} \\
  xy = xz\,, \quad xy\cdot z = xy\,.        \tag*{(U)}
\end{align*}
To show that a magma in $L_3$ lies in $U$, it is sufficient to show $xy=xx$ for all $x,y$, which is equivalent to $xy=xz$ for all $x,y,z$.

\begin{theorem}\label{Thm:L3_U}
A magma belongs to $U$ if and only if it belongs to $L_3$ and avoids each of the magmas $H_1$--$H_9$ with the following
Cayley tables:
\[
\begin{array}{c|cccc}
H_1 & 0 & 1 & 2 & 3 \\
\hline
0 & 2 & 3 & 2 & 2 \\ % 2,3,2,2,
1 & 2 & 2 & 2 & 2 \\ % 2,2,2,2,
2 & 2 & 2 & 2 & 2 \\ % 2,2,2,2
3 & 3 & 3 & 3 & 3    % 3,3,3,3
\end{array}
\qquad
\begin{array}{c|cccc}
H_2 & 0 & 1 & 2 & 3 \\
\hline
0 & 3 & 2 & 3 & 3 \\ % 3,2,3,3,
1 & 3 & 2 & 2 & 2 \\ % 3,2,2,2,
2 & 2 & 2 & 2 & 2 \\ % 2,2,2,2
3 & 3 & 3 & 3 & 3    % 3,3,3,3
\end{array}
\]
\[
\begin{array}{c|cccc}
H_3 & 0 & 1 & 2 & 3 \\
\hline
0 & 3 & 2 & 3 & 3 \\ % 3,2,3,3,
1 & 2 & 3 & 3 & 3 \\ % 2,3,3,3,
2 & 2 & 2 & 2 & 2 \\ % 2,2,2,2
3 & 3 & 3 & 3 & 3    % 3,3,3,3
\end{array}
\qquad
\begin{array}{c|cccc}
H_4 & 0 & 1 & 2 & 3 \\
\hline
0 & 2 & 2 & 2 & 2 \\ % 2,2,2,2,
1 & 2 & 3 & 3 & 3 \\ % 2,3,3,3,
2 & 2 & 2 & 2 & 2 \\ % 2,2,2,2
3 & 3 & 3 & 3 & 3    % 3,3,3,3
\end{array}
\]
\[
\begin{array}{c|ccccc}
H_5 & 0 & 1 & 2 & 3 & 4 \\
\hline
0 & 4 & 2 & 4 & 4 & 4 \\ % 4,2,4,4,4
1 & 3 & 2 & 2 & 2 & 2 \\ % 3,2,2,2,2
2 & 2 & 2 & 2 & 2 & 2 \\ % 2,2,2,2,2
3 & 3 & 3 & 3 & 3 & 3 \\ % 3,3,3,3,3
4 & 4 & 4 & 4 & 4 & 4    % 4,4,4,4,4
\end{array}
\qquad
\begin{array}{c|ccccc}
H_6 & 0 & 1 & 2 & 3 & 4 \\
\hline
0 & 3 & 2 & 2 & 2 & 2 \\ % 3,2,2,2,2
1 & 2 & 4 & 4 & 4 & 4 \\ % 2,4,4,4,4
2 & 2 & 2 & 2 & 2 & 2 \\ % 2,2,2,2,2
3 & 3 & 3 & 3 & 3 & 3 \\ % 3,3,3,3,3
4 & 4 & 4 & 4 & 4 & 4    % 4,4,4,4,4
\end{array}
\]
\[
\begin{array}{c|ccccc}
H_7 & 0 & 1 & 2 & 3 & 4 \\
\hline
0 & 4 & 2 & 4 & 4 & 4 \\ % 4,2,4,4,4
1 & 3 & 4 & 4 & 4 & 4 \\ % 3,4,4,4,4
2 & 2 & 2 & 2 & 2 & 2 \\ % 2,2,2,2,2
3 & 3 & 3 & 3 & 3 & 3 \\ % 3,3,3,3,3
4 & 4 & 4 & 4 & 4 & 4    % 4,4,4,4,4
\end{array}
\qquad
\begin{array}{c|ccccc}
H_8 & 0 & 1 & 2 & 3 & 4 \\
\hline
0 & 2 & 2 & 2 & 2 & 2 \\ % 2,2,2,2,2
1 & 3 & 4 & 4 & 4 & 4 \\ % 3,4,4,4,4
2 & 2 & 2 & 2 & 2 & 2 \\ % 2,2,2,2,2
3 & 3 & 3 & 3 & 3 & 3 \\ % 3,3,3,3,3
4 & 4 & 4 & 4 & 4 & 4    % 4,4,4,4,4
\end{array}
\]
\[
\begin{array}{c|cccccc}
H_9 & 0 & 1 & 2 & 3 & 4 & 5 \\
\hline
0 & 4 & 2 & 4 & 4 & 4 & 4 \\ % 4,2,4,4,4,4
1 & 3 & 5 & 5 & 5 & 5 & 5 \\ % 3,5,5,5,5,5
2 & 2 & 2 & 2 & 2 & 2 & 2 \\ % 2,2,2,2,2,2
3 & 3 & 3 & 3 & 3 & 3 & 3 \\ % 3,3,3,3,3,3
4 & 4 & 4 & 4 & 4 & 4 & 4 \\ % 4,4,4,4,4,4
5 & 5 & 5 & 5 & 5 & 5 & 5    % 5,5,5,5,5,5
\end{array}
\]
\end{theorem}
\begin{proof}
Let $M$ be a magma in $L_3$.  For $x,y\in M$, let $S(x,y) = \{x,y,xy,yx,xx,yy\}$. Whether or not the elements of $S(x,y)$ are distinct, $S(x,y)$ is a submagma of
$M$ with Cayley table
\begin{equation}\label{Eqn:L3_U}
\begin{array}{c|cccccc}
\cdot & x & y & xy & yx & xx & yy \\
\hline
 x & xx & xy & xx & xx & xx & xx \\
 y & yx & yy & yy & yy & yy & yy \\
xy & xy & xy & xy & xy & xy & xy \\
yx & yx & yx & yx & yx & yx & yx \\
xx & xx & xx & xx & xx & xx & xx \\
yy & yy & yy & yy & yy & yy & yy
\end{array}
\end{equation}
The entries of \eqref{Eqn:L3_U} are easily verified using \ref{L3}.

If \ref{U} fails in $M$, then there exist $a,b\in M$ such that $ab\neq aa$. We now show that $a$ and $b$ do not coincide with any other element of $S(a,b)$.
Clearly $a\neq b$. If $a = uv$ for some $u,v\in M$, then $ab = uv\cdot b = uv = uv\cdot a = aa$ by \ref{L3}. If $b = uv$ for some $u,v\in M$, then 
$ab = a\cdot uv = aa$ by \ref{L3}. Each case contradicts the assumption that $ab\neq aa$. We may also conclude that $a$, $b$, $ab$, and $aa$ are all distinct.

What remains is to consider the different cases of elements of $S(a,b)$ coinciding.

If no two elements coincide, then an examination of \eqref{Eqn:L3_U}
(with $x=a$, $y=b$) shows that $S(a,b)$ is a copy of $H_9$ with $a=0$, $b=1$, $ab=2$, $ba=3$, $aa=4$, $bb=5$.

If $ab=ba$ and all other elements are distinct, then \eqref{Eqn:L3_U} becomes
\[
\begin{array}{c|ccccc}
\cdot & a & b & ab & aa & bb \\
\hline
 a & aa & ab & aa & aa & aa \\
 b & ab & bb & bb & bb & bb \\
ab & ab & ab & ab & ab & ab \\
aa & aa & aa & aa & aa & aa \\
bb & bb & bb & bb & bb & bb
\end{array}
\]
This is evidently a copy of $H_6$ with $a=0$, $b=1$, $ab=2$, $aa=3$, $bb=4$. 

If $ab=bb$ and all other elements are distinct, then \eqref{Eqn:L3_U} becomes
\[
\begin{array}{c|ccccc}
\cdot & a & b & ab & ba & aa \\
\hline
 a & aa & ab & aa & aa & aa \\
 b & ba & bb & ab & ab & ab \\
ab & ab & ab & ab & ab & ab \\
ba & ba & ba & ba & ba & ba \\
aa & aa & aa & aa & aa & aa
\end{array}
\]
This is evidently a copy of $H_5$ with $a=0$, $b=1$, $ab=2$, $ba=3$, $aa=4$. 

If $aa=ba$ and all other elements are distinct, then it turns out that $S(a,b)$ is another copy of $H_5$, this time with
$b=0$, $a=1$, $aa=2$, $ab=3$, $bb=4$.

If $ba=bb$ and all other elements are distinct, then after permuting rows and columns, \eqref{Eqn:L3_U} becomes
\[
\begin{array}{c|ccccc}
\cdot & b & a & ba & ab & aa \\
\hline
 b & ba & ba & ba & ba & ba \\
 a & ab & aa & aa & aa & aa \\
ba & ba & ba & ba & ba & ba \\
ab & ab & ab & ab & ab & ab \\
aa & aa & aa & aa & aa & aa
\end{array}
\]
This is evidently a copy of $H_8$ with $b=0$, $a=1$, $ba=2$, $ab=3$, $aa=4$. 

If $aa=bb$ and all other elements are distinct, then \eqref{Eqn:L3_U} becomes
\[
\begin{array}{c|ccccc}
\cdot & a & b & ab & ba & aa \\
\hline
 a & aa & ab & aa & aa & aa \\
 b & ba & aa & aa & aa & aa \\
ab & ab & ab & ab & ab & ab \\
ba & ba & ba & ba & ba & ba \\
aa & aa & aa & aa & aa & aa
\end{array}
\]
This is evidently a copy of $H_7$ with $a=0$, $b=1$, $ab=2$, $ba=3$, $aa=4$. 

If $aa=ba=bb$ and all other elements are distinct, then after permuting rows and columns, \eqref{Eqn:L3_U} becomes
\[
\begin{array}{c|cccc}
\cdot & a & b & aa & ab \\
\hline
 a & aa & ab & aa & aa  \\
 b & aa & aa & aa & aa  \\
aa & aa & aa & aa & aa  \\
ab & ab & ab & ab & ab  
\end{array}
\]
This is evidently a copy of $H_1$ with $a=0$, $b=1$, $aa=2$, $ab=3$.

If $ab=ba=bb$ and all other elements are distinct, then after permuting rows and columns, \eqref{Eqn:L3_U} becomes
\[
\begin{array}{c|cccc}
\cdot & b & a & ab & aa \\
\hline
 b & ab & ab & ab & ab  \\
 a & ab & aa & aa & aa  \\
ab & ab & ab & ab & ab  \\
aa & aa & aa & aa & aa
\end{array}
\]
This is evidently a copy of $H_4$ with $a=0$, $b=1$, $aa=2$, $ab=3$. 

If $ab=ba$, $aa=bb$ and all other elements are distinct, then \eqref{Eqn:L3_U} becomes
\[
\begin{array}{c|cccc}
\cdot & a & b & ab & aa \\
\hline
 a & aa & ab & aa & aa  \\
 b & ab & aa & aa & aa  \\
ab & ab & ab & ab & ab  \\
aa & aa & aa & aa & aa
\end{array}
\]
This is evidently a copy of $H_3$ with $a=0$, $b=1$, $ab=2$, $aa=3$.

If $ab=bb$, $aa=ba$ and all other elements are distinct, then \eqref{Eqn:L3_U} becomes
\[
\begin{array}{c|cccc}
\cdot & a & b & ab & aa \\
\hline
 a & aa & ab & aa & aa  \\
 b & aa & ab & ab & ab  \\
ab & ab & ab & ab & ab  \\
aa & aa & aa & aa & aa
\end{array}
\]
This is evidently a copy of $H_2$ with $a=0$, $b=1$, $ab=2$, $aa=3$.

This exhausts all possibilities. Thus if \ref{U} fails in $M$, then $M$ contains a copy of at least one of $H_1$--$H_9$.
Conversely, if $M$ contains a copy of at least one of $H_1$--$H_5$, then in all $9$ cases, $1\cdot 0\neq  1\cdot 1$,
and so \ref{U} fails.
\end{proof}

\subsection{$U$ characterized inside $L_4$}
Recall the relevant identities:
\begin{align*}
  xy\cdot z = xz\,,\quad  x\cdot yz = xx \tag*{(L$_4$)} \\
  xy = xz\,,\quad xy\cdot z = xy         \tag*{(U)}
\end{align*}

\begin{theorem}\label{Thm:L4_U}
A magma in $L_4$ belongs to $U$ if and only if it avoids the magma $D$ with Cayley table
\[
\begin{array}{c|cccc}
D & 0 & 1 & 2 & 3 \\
\hline
0 & 0 & 2 & 0 & 0 \\
1 & 3 & 3 & 3 & 3 \\
2 & 0 & 2 & 0 & 0 \\
3 & 3 & 3 & 3 & 3
\end{array}
\]
\end{theorem}
\begin{proof}
Let $M$ be a magma in $L_4$. For all $x,y\in M$, let $S(x,y) = \{xx,y,xy,xy\}$. Whether the elements of $S(x,y)$ are
distinct or not, $S(x,y)$ is a submagma of $M$ with Cayley table
\begin{equation}\label{Eqn:inner-D}
\begin{array}{c|cccc}
\cdot & xx & y & xy & yy \\
\hline
xx & xx & xy & xx & xx \\
y  & yy & yy & yy & yy \\
xy & xx & xy & xx & xx \\
yy & yy & yy & yy & yy
\end{array}
\end{equation}
Indeed, for any $z,u,v\in M$, $xz\cdot uv\byref{L4} x\cdot uv\byref{L4} xx$ and $xz\cdot y\byref{L4} xy$. This gives the first and third rows of \eqref{Eqn:inner-D}.
For the second row, $y\cdot uv\byref{L4} yy$ and for the fourth row, we have $yy\cdot uv\byref{L4}y\cdot uv\byref{L4} yy$ and $yy\cdot y\byref{L4} yy$.

Now if $U$ fails to hold in $M$, then for some $a,b\in M$, $ab\neq aa$. We now show that $|S(a,b)|=4$. Indeed, if $aa=b$, then $ab = a\cdot aa\byref{L4} aa$. If
$aa=bb$, then $aa\cdot a= bb\cdot a$, which by \ref{L4} implies $aa=ba$; this, in turn, implies $aa\cdot b = ba\cdot b$, which by \ref{L4} implies
$ab=bb=aa$. If $b=ab$, then $ab=a\cdot ab\byref{L4} aa$. If $b=bb$, then $ab=a\cdot bb\byref{L4} aa$. Finally, if $ab=bb$, then $ab\cdot a=bb\cdot a$, which by \ref{L4}
implies $aa=ba$, a case we have already considered. Thus all elements of $S(a,b)$ are distinct.

Thus if $U$ fails to hold in $M$, $S(a,b)$ is a copy of $D$ (with $aa=0$, $b=1$, $ab=2$, and $bb=3$). Conversely, if $M$ contains a copy of $D$, then $0\cdot 1 = 2\neq 0
=0\cdot 0$, and thus $U$ fails in $M$.
\end{proof}

\subsection{$U$ characterized inside $L_5$}
Recall the relevant identities
\begin{align*}
  xy\cdot z = xy\,,\quad x\cdot yz = xy \tag*{(L$_5$)} \\
  xy = xz\,,\quad  xy\cdot z = xy       \tag*{(U)}
\end{align*}
Since $L_5$ is a variety of semigroups, we will use associativity freely. To show that a magma in $L_5$ belongs
to $U$, it is sufficient to show $xy=xz$ for all $x,y,z$, or equivalently, $xy=xx$ for all $x,y$.

\begin{theorem}\label{Thm:L5_U}
A magma in $L_5$ belongs to $U$ if and only if it avoids the magma $P$ with Cayley table
\[
\begin{array}{c|cccc}
P & 0 & 1 & 2 & 3 \\
\hline
0 & 2 & 3 & 2 & 2 \\
1 & 1 & 1 & 1 & 1 \\
2 & 2 & 2 & 2 & 2 \\
3 & 3 & 3 & 3 & 3
\end{array}
\]
\end{theorem}
\begin{proof}
Let $M$ be a magma in $L_5$. For all $x,y\in M$, let $S(x,y)=\{x,yx,xx,xy\}$. Whether the elements of $S(x,y)$ are
distinct or not, $S(x,y)$ is a submagma of $M$ with Cayley table
\begin{equation}\label{Eqn:inner-P}
\begin{array}{c|cccc}
\cdot & x & yx & xx & xy \\
\hline
 x & xx & xy & xx & xx \\
yx & yx & yx & yx & yx \\
xx & xx & xx & xx & xx \\
xy & yy & yy & yy & yy
\end{array}
\end{equation}
Each of the entries of this table follows immediately from \ref{L5}.

If \ref{U} fails to hold in $M$, then for some $a,b\in M$, $ab\neq aa$. We now show that $|S(a,b)|=4$. Indeed, if $a=ba$, then $aa=aba\byref{L5}ab$.
If $a=aa$, then $ab=aab\byref{L5}aa$. If $a=ab$, then $aa=aba\byref{L5}ab$. If $ba=aa$, then $aba=aaa$ which by \ref{L5} implies $ab=aa$.
Finally, if $ba=ab$, then $aba=aab$ which by \ref{L5} implies $ab=aa$. Thus all elements of $S(a,b)$ are distinct.

Thus if $U$ fails to hold in $M$, $S(a,b)$ is a copy of $P$ (with $a=0$, $ba=1$, $aa=2$, $ab=3$). Conversely, if $M$ contains a copy of $P$, then
$0\cdot 1=3\neq 2=0\cdot 0$, and thus $U$ fails in $M$.
\end{proof}

\subsection{$U$ characterized inside $L_6$}
Recall the relevant identities:
\begin{align*}
  xy\cdot z = xz\,,\quad  x\cdot yz = xy \tag*{(L$_6$)} \\
  xy = xz\,,\quad xy\cdot z = xy         \tag*{(U)}
\end{align*}
To prove that a magma in $L_6$ lies in $U$, it is sufficient to show $xy=xx$ for all $x,y$, which implies
$xy=xz$ for all $x,y$, and then $xy\cdot z\byref{L6} xz = xy$.

\begin{theorem}\label{Thm:L6_U}
A magma in $L_6$ belongs to $U$ if and only if it avoids the magmas $M_1$ and $M_2$ with the following
Cayley tables:
\[
\begin{array}{c|ccc}
M_1 & 0 & 1 & 2 \\
\hline
0 & 0 & 2 & 0 \\
1 & 1 & 1 & 1 \\
2 & 0 & 2 & 0
\end{array}
\qquad
\begin{array}{c|cccc}
M_2 & 0 & 1 & 2 & 3 \\
\hline
0 & 0 & 2 & 0 & 2 \\
1 & 1 & 3 & 1 & 3 \\
2 & 0 & 2 & 0 & 2 \\
3 & 1 & 3 & 1 & 3
\end{array}
\]
\end{theorem}
\begin{proof}
  Let $M$ be a magma in $L_6$. For $x,y\in M$, let $S(x,y) = \{xx,yx,xy,yy\}$. Whether the elements of $S(x,y)$ are
distinct or not, $S(x,y)$ is a submagma of $M$ with Cayley table
\begin{equation}\label{Eqn:inner-Mi}
\begin{array}{c|cccc}
\cdot & xx & yx & xy & yy \\
\hline
xx & xx & xy & xx & xx \\
yx & yx & yy & yx & yy \\
xy & xx & xy & xx & xy \\
yy & yx & yy & yx & yy
\end{array}
\end{equation}
Each entry of \eqref{Eqn:inner-Mi} is easily checked using \ref{L6}.

If \ref{U} fails to hold in $M$, then for some $a,b\in M$, $ab\neq aa$. This turns out to imply $|S(a,b)|\geq 3$.
Indeed, if $au=bv$ for some $u,v\in \{a,b\}$, then $a\cdot au=a\cdot bv$, and so $aa=ab$ by \ref{L6}. Thus there
are two cases to consider.

If $ba\neq bb$, then $|S(a,b)| = 4$ and $S(a,b)$ is a copy of $M_2$ with $aa=0$, $ba=1$, $ab=2$, $bb=3$.

If $ba = bb$, then $|S(a,b)|=3$ and $S(a,b)$ is a copy of $M_1$ with $aa=0$, $ba=1$, $ab=2$:
\[
\begin{array}{c|ccc}
\cdot & aa & ba & ab \\
\hline
aa & aa & ab & aa \\
ba & ba & ba & ba \\
ab & aa & ab & aa
\end{array}
\]

Thus if $M$ does not belong to $U$, then $M$ contains a copy of $M_1$ or $M_2$. Conversely, if $M$ contains a copy of
$M_1$ or of $M_2$, then $0\cdot 1 = 2\neq 0=0\cdot 0$, and thus in either case, \ref{U} fails to hold.
\end{proof}

\subsection{$U$ characterized inside $L_7$}
Recall the relevant identities:
\begin{align*}
    xy\cdot z=xz,\quad x\cdot yz=xz  \tag*{(L$_7$)} \\
    xy=xz,\quad  xy\cdot z=xy        \tag*{(U)}
\end{align*}
Since $L_7$ is a variety of semigroups, we use associativity freely. To prove that a semigroup in $L_7$ lies in $U$, it is
sufficient to show $xy=xx$ for all $x,y$, for then $xy=xz$ for all $x,y,z$, and then $xyz\byref{L7} xz = xy$.

\begin{theorem}
A magma belongs to $U$ if and only if it belongs to $L_7$ and avoids the $2$-element right zero band $2_{RZ}$.
\end{theorem}
\begin{proof}
Let $M$ be a magma in $L_7$. For $x,y\in M$, let $S(x,y) = \{xx,xy\}$. Whether the elements of $S(x,y)$ are
distinct or not, $S(x,y)$ is a submagma of $M$ and a right zero band, with Cayley table
\begin{equation}\label{Eqn:rzb_L7}
\begin{array}{c|cc}
\cdot & xx & xy \\
\hline
xx & xx & xy \\
xy & xx & xy
\end{array}
\end{equation}
The entries of \eqref{Eqn:rzb_L7} are easily confirmed with repeated use of \ref{L7}.

Now if $M$ does not belong to $U$, then there exist $a,b\in M$ such that $ab\neq aa$. In this case, $|S(a,b)|=2$
and $S(a,b)$ is a copy of $2_{RZ}$. Conversely, if $M$ contains a copy of $2_{RZ}$, then $0\cdot 1=1\neq 0 = 0\cdot 0$,
and so \ref{U} fails.
\end{proof}

\section{The computational tool}\label{tool}

The tool we used to produce these theorems can be found in \cite{site}. It produces inputs that then can be fed into ProverX \cite{yves,proverx} to get the proofs.

In general terms, the  algorithm in \cite{site} is the following: the input are  two classes of algebras  $A\subset B$; as the classes $A$ and $B$ are different, there exists one model $M$ in $B\setminus A$; let $\mathcal{F}:=\{M\}$ and ask the computer whether  
\begin{eqnarray}\label{1}
A\stackrel{?}{=}\llbracket B\mid \mathcal{F} \rrbracket\, 
\end{eqnarray}

One of the following holds: 
\begin{enumerate}
\item[(1)] there is a proof, and the user gets the desired theorem; 
\item[(2)] there is a counter-example; in this case we add it to $\mathcal{F}$ and ask (\ref{1}) again;
\item[(3)] there is neither a proof nor a counter-example.
\end{enumerate}
 
Occasionally, hundreds of models appear at step (2). In some cases, these models are well behaved (e.g., it is one per order, but they seem to increase without limit) and this possibly hints at an infinite family of forbidden substructures; in other cases, the hundreds of models look very different, possibly an hint that a forbidden substructure theorem is not the right approach.

On the other extreme, occasionally the tool will find a small set of  models, and then will be searching forever for a proof or counter-example. We take this as an hint that probably there is an infinite forbidden substructure. Right now the most promising way of trying to improve \cite{site} is to automatically  find ways of getting finite descriptions of these infinite models in such a way that we still can get a proof.

%
%\begin{enumerate}
%\item let $\mathcal{F} :=\emptyset$;
%\item\label{cycle} find a model $M$ in $B\setminus A$ such that no model in $\mathcal{F} $ embeds in $M$;
%\item add $M$ to $\mathcal{F} $;
%\item  if the computer finds a proof for $A=\llbracket B,\mathcal{F} \rrbracket$, then stop; otherwise go to \ref{cycle}.
%\end{enumerate}
%
%Of course the computer's success finding proofs depends heavily on the input. A first algorithm to write inputs managed to produce automatically almost all theorems in this paper, but not all (chiefly, $U$ inside $L_3$ failed). Therefore we devised a new and more powerful algorithm that is called whenever the first one fails. It might happen that for a certain class of algebras, the current algorithms do not produce the desired result. In that case, please send us an email because we might be able to sort the case and eventually improve the tool.

The version in \cite{site} is very user friendly; the code is also included in \cite{proverx} so that advanced users can control it using Python. 

The details of this tool (from a computer science point of view) will appear elsewhere. It does not seem necessary to guarantee here that the algorithms really work as the output is a proof that a mathematician can check (as happens  with the results in this paper). 

Whenever you use this tool, please cite it as that is important to help us get the funds needed to keep improving it. If you have any suggestions of improvements, please let us know.

\end{document}